\documentclass[a4paper,11pt]{amsart}

\usepackage{amsmath,amssymb,amsthm,mathtools,hyperref,geometry,cleveref,xcolor}
\allowdisplaybreaks

\numberwithin{equation}{section}
\theoremstyle{plain}
\newtheorem{thm}{Theorem}[section]

\newtheorem{lem}[thm]{Lemma}

\newtheorem{conj}[thm]{Conjecture}

\newtheorem{rem}[thm]{Remark}
\newtheorem*{acknow}{Acknowledgments}

\def\tri{\triangle}

\def\l{\langle}
\def\r{\rangle}

\def\B{\mathcal{B}}
\def\C{\mathcal{C}}
\def\ll{\left}
\def\rr{\right}
\def\nn{\nabla}
\def\u{\mathfrak{u}}

\title[Lu's conjecture for minimal surfaces]{Lu's conjecture for minimal surfaces}

\author[W. R. Ding]{Weiran Ding$^{1}$}
\address{$^{1}$School of Mathematical Sciences, South China Normal University, Guangzhou, Guangdong 510000, P. R. CHINA.}
\email{dingwr0806@m.scnu.edu.cn}

\author[J. Q. Ge]{Jianquan Ge$^{2}$}
\address{$^{2}$School of Mathematical Sciences, Laboratory of Mathematics and Complex Systems, Beijing Normal University, Beijing 100875, P. R. CHINA.}
\email{jqge@bnu.edu.cn}

\author[F. G. Li]{Fagui Li$^{3,*}$}
\address{$^{3,*}$Frontier Interdisciplinary Domain, Beijing Institute of Technology, Zhuhai, Guangdong 519088, P. R. CHINA.}
\email{lifagui@bitzh.edu.cn}

\author[X. Z. Yang]{Xize Yang$^{2}$}
\address{$^{2}$School of Mathematical Sciences, Laboratory of Mathematics and Complex Systems, Beijing Normal University, Beijing 100875, P. R. CHINA.}
\email{yangxize@mail.bnu.edu.cn}

\date{}
\thanks{* the corresponding author.}
\thanks{J. Q. Ge is partially supported by NSFC (No. 12571049) and the Fundamental Research Funds for the Central Universities.}
\thanks{F. G. Li is partially supported by  NSFC (No. 12271040, 12501061) and Research Start up Funding of Beijing Institute of Technology (No. 5640011253301).}
\begin{document}
\maketitle

\begin{abstract}
	After Chern's conjecture on the discreteness of the constant scalar curvatures of compact minimal submanifolds $M^n$ in unit spheres $\mathbb{S}^{n+q}$, Z. Q. Lu proposed a conjecture regarding the second gap, based on his ingenious refinement of the known first gap theorem. This refinement unifies Simons' first gap theorem for hypersurfaces with the corresponding theorems for high-codimensional submanifolds established by Yau, Shen, Li and Li, among others. In this paper, for arbitrary codimension, we prove Lu's conjecture for minimal 2-spheres, and for any minimal surfaces under some slight inequality conditions about the normal scalar curvature.
\end{abstract}

\section{Introduction}
Let $f:M^n\to\tilde{M}^{n+q}(c)$ be an isometric immersion of an $n$-dimensional submanifold $M$ into the $(n+q)$-dimensional real space form $\tilde{M}^{n+q}(c)$ of constant sectional curvature $c$. The normalized scalar curvature $\rho$ and normal scalar curvature $\rho^\perp$ are defined by 
\begin{align*}
	\rho&=\frac{2}{n(n-1)}\sum_{1=i<j}^{n}R(e_i,e_j,e_j,e_i),\\
	\rho^\perp&=\frac{1}{n(n-1)}\left|R^\perp\right|=\frac{2}{n(n-1)}\left(\sum_{1=i<j}^{n}\sum_{1=r<s}^q\left\langle R^\perp(e_i,e_j)\xi_r,\xi_s\right\rangle^2\right)^{1/2},
\end{align*}
where $\{e_1,\cdots,e_n\}$ is an orthonormal basis of the tangent space, and $R$ is the curvature tensor of the tangent bundle. Similarly, $\{\xi_1,\cdots,\xi_q\}$ is an orthonormal basis of the normal space, and $R^\perp$ is the curvature tensor of the normal bundle. Let $h$ be the second fundamental form and let $H=\frac{1}{n}\operatorname{Tr}h$ be the mean curvature vector field. The DDVV conjecture, due to De Smet, Dillen, Verstraelen, and Vrancken \cite{DDVV}, asserts the following pointwise inequality:
\begin{equation}\label{ddvv1}
	\rho+\rho^\perp\le|H|^2+c.
\end{equation}
Since \eqref{ddvv1} is a pointwise inequality among $\rho,\rho^\perp$, and $|H|^2$, it follows from the Gauss equation and the Ricci equation that it is equivalent to the following algebraic inequality \cite{Dillen}:
\begin{equation*}
	\sum_{r,s=1}^q\|[B_r,B_s]\|^2\le\left(\sum_{r=1}^q\|B_r\|^2\right)^2,
\end{equation*} 
where $B_1,\cdots,B_q$ are $n\times n$ real symmetric matrices, $\|\cdot\|^2$ denotes the sum of the squares of the entries of the matrix and $[A,B]=AB-BA$ is the commutator of the matrices $A$ and $B$. The first nontrivial case, $n=3$, was established by Choi and Lu \cite{Choi-Lu}. A weaker version for general $n$ and $q$ was later obtained in \cite{Dillen}. Finally, the DDVV conjecture was independently proved by Ge and Tang \cite{Ge-Tang} and by Lu \cite{Lu}. For further details on the DDVV-type inequalities, we refer the reader to \cite{GLTZ,GLZ}. Building on the methods developed in the proof of the DDVV conjecture, Lu derived several pinching theorems and proposed the following conjecture:
\begin{conj}[Lu's Conjecture \cite{Lu}]\label{Lu}
	Let $M^n$ be a closed immersed minimal submanifold of the unit sphere $\mathbb{S}^{n+q}$. Let $\lambda_1\ge\lambda_2\ge\cdots\ge\lambda_q$ be the eigenvalues of the fundamental matrix $A=(a_{\alpha\beta})=(\l S_\alpha,S_\beta\r)$, where $\{S_{\alpha}\}_{\alpha=1}^q$ are shape operators with respect to an orthonormal basis $\{\xi_\alpha\}_{\alpha=1}^q$. If $S+\lambda_2$ is a constant and if $$S+\lambda_2>n,$$ then there is a constant $\epsilon(n,q)>0$ such that $$S+\lambda_2>n+\epsilon(n,q).$$
\end{conj}
\Cref{Lu} is based on the following rigidity theorem of Lu \cite{Lu}: \emph{Let $M^n$ be a closed minimal submanifold of the unit sphere $\mathbb{S}^{n+q}$. If $0\leq S+\lambda_2\leq n$, then $M$ is totally geodesic, or one of the Clifford tori $M_{k,n-k}$ $(1\leq k<n)$ in $\mathbb{S}^{n+q}$, or a Veronese surface in $\mathbb{S}^{2+q}$.} The above pinching theorem generalized the earlier pinching theorems established by Simons \cite{Simons}, Chern, do Carmo, and Kobayashi \cite{CDK}, Yau \cite{Yau 1974,Yau 1975}, Shen \cite{YBShen 1989}, as well as Li and Li \cite{LiLi 1992}, etc. Recently, Leng and Xu \cite{LX18} generalized Lu's rigidity theorem to submanifolds with parallel mean curvature. Furthermore, \Cref{Lu} can be viewed as a high-codimensional generalization of a classical gap theorem established by Peng and Terng \cite{PengT}. Indeed, the case $q=1$ was solved by Peng and Terng \cite{PengT}, which is related to the famous Chern conjecture \cite{CDK,TWY,TY}: \emph{Let $M^n$ be a closed immersed minimal hypersurface of the unit sphere $\mathbb{S}^{n+1}$ with constant scalar curvature $R_M$. Then for each $n$, the set of all possible values for $R_M$ is discrete.} Despite considerable effort, a complete proof of the Chern conjecture remains elusive, though partial results are known in low dimensions or under additional curvature assumptions. The first gap theorem was established by Simons \cite{Simons}, who showed that if $0\le S\le n$, then either $S=0$ or $S=n$ identically on $M$. More recently, it was shown that for a certain class of austere submanifolds, the assumption $0<S\le n$ forces the submanifold to be a Clifford torus \cite{GeTaoZhou}. The high-codimensional formulation of Chern's conjecture likewise remains an open problem. For 2-dimensional minimal surfaces of constant curvature, Chern's conjecture was established by Calabi \cite{Calabi} for minimal 2-spheres and by Bryant \cite{Bryant} for general surfaces. To the best of our knowledge, for high-codimensional submanifolds of dimension three or higher, both Chern's conjecture and Lu's conjecture remain unresolved regarding the second gap. Subsequent work over several decades has addressed the second gap phenomenon, leading to numerous significant contributions (cf. \cite{Ding-Xin,LeiXuXu,PengT,XuXu,Yang-Cheng3}). We refer to the surveys \cite{GuXuXu,Yau} for a comprehensive overview and further references. Another famous rigidity problem, closely related to both the Lu conjecture and the Chern conjecture, was proposed by Simon \cite{Simon2}:
\begin{conj}[Simon's Conjecture \cite{Simon2}]\label{Udo Simon}
	Let $M$ be a closed surface minimally immersed in $\mathbb{S}^N$ such that the image is not contained in any hyperplane of $\mathbb{R}^{N+1}$. If $K(s+1)\le K\le K(s)$ for an $s\in\mathbb{N}$, where $K$ is the curvature of $M$ and $K(s)\coloneqq\frac{2}{s(s+1)}$, then either $K=K(s+1)$ or $K=K(s)$, and thus the submanifold is one of Calabi's $2$-spheres \emph{(}\cite{Calabi}, see also \Cref{thm Calabi}\emph{)} with the dimension of the ambient space $N=2s+2$ or $N=2s$, respectively.
\end{conj}
For minimal surfaces in $\mathbb{S}^N$, the curvature $K$ and the squared norm $S=|h|^2$ of the second fundamental form $h$ are related by $2K=2-S$. Therefore, the rigidity of $K$ is equivalent to the rigidity of $S$ in \Cref{Udo Simon}. So far, the Simon conjecture has only been solved in the cases $s=1$ and $s=2$ (cf. \cite{BKSS,DGL,Simon}). For the case $s\ge 3$,  we refer the reader to the review article \cite{XuXW2025} for further details.

In this article, using the result of the Simon conjecture in the case $s=2$, we firstly answer affirmatively Lu's conjecture for 2-spheres, even with an optimal second gap prescribed by Calabi's 2-spheres.
\begin{thm}\label{main1}
	Let $M^2$ be a $2$-sphere minimally immersed in $\mathbb{S}^{2+q}$.
	\begin{enumerate}
		\item\label{main1.1} If $S+\lambda_2$ is  constant, then the curvature $K$ is constant and the submanifold is one of Calabi's $2$-spheres.
		\item\label{main1.2} If $S+\lambda_2>2$, then $$\max_{p\in M} \left( S+\lambda_2\right) (p)\ge\frac{5}{2},$$ where the equality holds if and only if $q=4$ and the submanifold is Calabi's $2$-sphere with curvature $K=\frac{1}{6}$.
	\end{enumerate}	 
\end{thm}
For general surfaces, we introduce some pinching conditions of the normal scalar curvature and establish the following gap rigidity theorems.
\begin{thm}\label{main4}
	Let $M^2$ be a closed minimal surface immersed in $\mathbb{S}^{2+q}$ whose universal cover is not diffeomorphic to a $2$-sphere. Then one of the following two cases must occur:
	\begin{enumerate}
		\item $\max_{p\in M}(S+\lambda_2)(p)\ge\frac{8}{3}$;
		\item $\max_{p\in M}(S+\lambda_2)(p)\ge3-\sqrt{1-\min_{p\in M}\left(\rho^\perp\right)^2(p)}$ and $\min_{p\in M}\rho^\perp(p)\le 1$.
	\end{enumerate}
\end{thm}
\begin{thm}\label{main4.5}
	Let $M^2$ be a closed minimal surface immersed in $\mathbb{S}^{2+q}$ whose universal cover is not diffeomorphic to a $2$-sphere. We have
	\begin{equation*}
		\max_{p\in M}(S+\lambda_2)(p)\ge1+\sqrt{1+\frac{1}{\operatorname{Area}(M)}\int_M\left(\rho^\perp\right)^2}.
	\end{equation*}
\end{thm}
We also obtain the following integral formulas, which generalize the formulas in \cite{DGL}.
\begin{thm}\label{main3}
	Let $M^2$ be a closed minimal surface immersed in $\mathbb{S}^{2+q}$. Then
	\begin{equation}\label{1stgap}
		0\le\int_M\left[S(3S-4)-\left(S-2\lambda_2\right)^2\right]=2\int_M\mathcal{B}_1,
	\end{equation}
	and
	\begin{equation}\label{2ndgap}
		\begin{aligned}
			0&\le\int_M\left[S(3S-4)(3S-5)+\frac{1}{2}(16-9S)(S-2\lambda_2)^2\right]\\
			&=\int_M\ll[\frac{S}{2}(S-2)(9S-20)+2(\rho^\perp)^2(9S-16)\rr]\\
			&=2\int_M\ll[\mathcal{B}_2+\frac{1}{2}|\nabla S|^2-S(2-S)^2+(8-5S)(\rho^\perp)^2\rr],
		\end{aligned}
	\end{equation}
	where $\mathcal{B}_1=|\nabla h|^2$ and $\B_2=|\nn^2 h|^2$ are the squared lengths of the first and second covariant derivatives of $h$, respectively. 
\end{thm}
As applications of \Cref{main3}, we prove the following rigidity results.
\begin{thm}\label{main5}
	Let $M^2$ be a closed minimal surface immersed in $\mathbb{S}^{2+q}$.
	\begin{enumerate}
		\item\label{main5.1} If $2\le S\le S_0$, where $S_0=\frac{40-9\gamma+\sqrt{81\gamma^2+432\gamma+1600}}{36}$, and $\rho^\perp\le\frac{1}{2}\sqrt{\gamma|K|}$ for some $0\le\gamma\le 4$, then $S\equiv2$ and $M$ is the Clifford torus $\mathbb{S}^1(\sqrt{1/2})\times\mathbb{S}^1(\sqrt{1/2})$.
		\item\label{main5.2} If $S+\lambda_2>2$ and $\rho^\perp\le\frac{1}{2}\sqrt{(S+\lambda_2-2)\gamma S}$ for some $0\le\gamma\le\frac{2}{3}$, then $$\max_{p\in M}(S+\lambda_2)(p)\ge\max_{p\in M}S(p)>\frac{40+12\gamma}{18+9\gamma}\ge2.$$
	\end{enumerate}
\end{thm}
Define
\begin{equation*}
	\hat{T}_A(\tau)\coloneqq\frac{3-\tau}{2}\left(\frac{27-8\tau^2}{18-9\tau^2}+\frac{\sqrt{\left(8\tau^2-\frac{9}{2}\right)^2-\frac{45}{4}}}{18-9\tau^2}\right)\eqqcolon\frac{3-\tau}{2}T_A(\tau),
\end{equation*}
and
\begin{equation*}
	\hat{T}_B(\tau)\coloneqq\frac{3-\tau}{2}\left(\frac{27-8\tau^2}{18-9\tau^2}-\frac{\sqrt{\left(8\tau^2-\frac{9}{2}\right)^2-\frac{45}{4}}}{18-9\tau^2}\right)\eqqcolon\frac{3-\tau}{2}T_B(\tau).
\end{equation*}
A direct calculation yields that, for $\frac{\sqrt{9+3\sqrt{5}}}{4}\le\tau\le 1$, $\hat{T}_A(\tau)$ is a monotonically increasing function of $\tau$ and $\hat{T}_B(\tau)$ is a monotonically decreasing function of $\tau$. Moreover,  we have
\begin{equation*}
	\frac{30+2\sqrt{5}}{11}-\frac{\left(15+\sqrt{5}\right)\sqrt{9+3\sqrt{5}}}{66}=\hat{T}_A\left(\frac{\sqrt{9+3\sqrt{5}}}{4}\right)\le\hat{T}_A(\tau)\le\hat{T}_A(1)=\frac{20}{9},
\end{equation*}
and
\begin{equation*}
	2=\hat{T}_B(1)\le\hat{T}_B(\tau)\le\hat{T}_B\left(\frac{\sqrt{9+3\sqrt{5}}}{4}\right)=\frac{30+2\sqrt{5}}{11}-\frac{\left(15+\sqrt{5}\right)\sqrt{9+3\sqrt{5}}}{66}.
\end{equation*}
\begin{thm}\label{main6}
	Let $M^2$ be a closed minimal surface immersed in $\mathbb{S}^{2+q}$. 
	\begin{enumerate}
		\item\label{1.7.1} If $\rho^\perp\ge\frac{1}{2}\sqrt{1-\tau^2}S$ for some $0\le\tau\le 1$, then $$\max_{p\in M}(S+\lambda_2)(p)\ge(3-\tau)\left(1-\frac{2\pi\chi(M)}{\operatorname{Area}(M)}\right),$$ where $\chi(M)$ is the Euler characteristic, and $\operatorname{Area}(M)$ is the area of the surface $M$.
		\item\label{1.7.2} If $\rho^\perp\le\frac{1}{2}\sqrt{1-\tau^2}S$ and $S+\lambda_2>\hat{T}_B(\tau)$ for some $\frac{\sqrt{9+3\sqrt{5}}}{4}\le\tau\le 1$, then $$\max_{p\in M}(S+\lambda_2)(p)\ge\hat{T}_A(\tau)=\hat{T}_B(\tau)+\sigma(\tau),$$ where $\sigma(\tau)=\frac{3-\tau}{18-9\tau^2}\sqrt{\left(8\tau^2-\frac{9}{2}\right)^2-\frac{45}{4}}$. In particular, if the normal bundle of $M$ is flat (i.e., $\rho^\perp=0$) and $S+\lambda_2>2$, then $$\max_{p\in M}(S+\lambda_2)(p)>\frac{20}{9}.$$
	\end{enumerate}
\end{thm}
\begin{rem}
	We point out that, in \Cref{main6} \eqref{1.7.2}, if $\rho^\perp\le\frac{1}{2}\sqrt{1-\tau^2}S$ and $S+\lambda_2>\hat{T}_B(\tau)$, where $\frac{\sqrt{9+3\sqrt{5}}}{4}<0.991\le\tau\le 1$, then $$\max_{p\in M}(S+\lambda_2)(p)\ge\hat{T}_A(\tau)\ge\hat{T}_B(\tau)+\sigma(0.991)>\hat{T}_B(1)+0.02=2.02.$$
\end{rem}

\section{Notation and Local Formulas}
Let $M$ be a 2-dimensional manifold immersed in a unit sphere $\mathbb{S}^{2+q}$. We adopt the following index conventions:
\begin{align*}
	1\le i,j,k,&\cdots\le2,\\
	3\le\alpha,\beta,\gamma,&\cdots\le 2+q,\\
	1\le A,B,C,&\cdots\le 2+q.
\end{align*}
Let $(e_A)$ be a local orthonormal frame on $T(\mathbb{S}^{2+q})$ such that, when restricted to $M$, $(e_i)$ and $(e_\alpha)$ lie in the tangent bundle $T(M)$ and the normal bundle $T^{\bot}(M)$, respectively. We take $(\omega_A)$ and $(\omega_{AB})$ to be the metric 1-form field and connection form field associated with $(e_A)$. Let $$S_\alpha=(h_{ij}^\alpha)_{2\times 2},$$ where $\omega_{i\alpha}=h_{ij}^\alpha\omega_j$. Then we have $h_{ij}^\alpha=h_{ji}^\alpha$. In the following, we will use the Einstein summation convention. The second fundamental form of $M$ is defined by $h=h_{ij}^\alpha\omega_i\omega_j e_\alpha$. The mean curvature normal vector field $H$ is defined by $2H=h_{ii}^\alpha e_\alpha$. If the mean curvature normal vector field of $M$ vanishes identically, the immersion is called minimal. Now we consider the minimal case. Let $\mathcal{B}_1=|\nabla h|^2$ and $\mathcal{B}_2=|\nabla^2 h|^2$ be the squared lengths of the first and second covariant derivatives of $h$, respectively. Define column vectors
\begin{align*}
	a=(a^\alpha)&\coloneqq(h_{11}^\alpha)=(-h_{22}^\alpha)\in\mathbb{R}^{q},\\
	b=(b^\alpha)&\coloneqq(h_{12}^\alpha)=(h_{21}^\alpha)\in\mathbb{R}^{q},\\
	a_i=(a_i^\alpha)&\coloneqq(h_{11i}^{\alpha})\in\mathbb{R}^{q},\\
	a_{ij}=(a_{ij}^\alpha)&\coloneqq(h_{11ij}^{\alpha})\in\mathbb{R}^{q}.
\end{align*}
We use the following notation:
\begin{align*}
	A&\coloneqq(\l S_{\alpha},S_{\beta}\r)=2aa^T+2bb^T,\\
	S&\coloneqq\text{tr}A=|h|^2.
\end{align*}
For technical convenience, we define $$\rho_0^\perp\coloneqq\sum_{\alpha,\beta}\left|[S_{\alpha},S_{\beta}]\right|^2.$$ A straightforward calculation and the DDVV inequality \eqref{ddvv1} show that
\begin{equation}\label{rhorho0}
	\rho_0^\perp=(2\rho^\perp)^2\le S^2.
\end{equation}
The covariant derivatives $h_{ijk}^\alpha$, $h_{ijkl}^\alpha$ and $h_{ijklm}^\alpha$ are defined as follows:
\begin{align*}
	h_{ijk}^\alpha\omega_k&=dh_{ij}^\alpha+h_{mj}^\alpha\omega_{mi}+h_{im}^\alpha\omega_{mj}+h_{ij}^\beta\omega_{\beta\alpha},\\
	h_{ijkl}^\alpha\omega_l&=dh_{ijk}^\alpha+h_{mjk}^\alpha\omega_{mi}+h_{imk}^\alpha\omega_{mj}+h_{ijm}^\alpha\omega_{mk}+h_{ijk}^\beta\omega_{\beta\alpha},\\
	h_{ijklm}^\alpha\omega_m&=dh_{ijkl}^\alpha+h_{njkl}^\alpha\omega_{ni}+h_{inkl}^\alpha\omega_{nj}+h_{ijnl}^\alpha\omega_{nk}+h_{ijkn}^\alpha\omega_{nl}+h_{ijkl}^\beta\omega_{\beta\alpha}.
\end{align*}
The Codazzi equation and the Ricci formulas are
\begin{align*}
	h_{ijk}^\alpha-h_{ikj}^\alpha&=0,\\
	h_{ijkl}^\alpha-h_{ijlk}^\alpha&=h_{pj}^\alpha R_{pikl}+h_{ip}^\alpha R_{pjkl}+h_{ij}^\beta R^{\perp}_{\beta\alpha kl},\\
	h_{ijklm}^\alpha-h_{ijkml}^\alpha&=h_{pjk}^\alpha R_{pilm}+h_{ipk}^\alpha R_{pjlm}+h_{ijp}^\alpha R_{pklm}+h_{ijk}^\beta R^{\perp}_{\beta\alpha lm}.
\end{align*}
The Laplacians $\tri h_{ij}^\alpha$ and $\tri h_{ijk}^\alpha$ are defined by
\begin{equation*}
	\tri h_{ij}^\alpha=\sum_{k}h_{ijkk}^\alpha\quad\text{and}\quad\tri h_{ijk}^\alpha=\sum_{l}h_{ijkll}^\alpha,
\end{equation*}
respectively, from which we obtain
\begin{equation}\label{Lap1}
	\tri h_{ij}^\alpha =h_{mmij}^\alpha+h_{pi}^\alpha R_{pmjm}+h_{mp}^\alpha R_{pijm}+h_{mi}^\delta R^{\perp}_{\delta\alpha jm},
\end{equation}
and
\begin{equation}\label{Lap2}
	\begin{aligned}
		\tri h_{ijk}^\alpha &=(\tri h_{ij}^\alpha)_k+2h_{pjm}^\alpha R_{pikm}+2h_{ipm}^\alpha R_{pjkm}+h_{ijp}^\alpha R_{pmkm}\\
	&\hspace{1.3em}+h_{pj}^\alpha R_{pikmm}+h_{ip}^\alpha R_{pjkmm}+2h_{ijm}^\delta R^{\perp}_{\delta\alpha km}+h_{ij}^\delta R^{\perp}_{\delta\alpha kmm}.
	\end{aligned}	
\end{equation}
The Riemannian curvature tensor, the normal curvature tensor and first covariant derivative of the normal curvature tensor are given by
\begin{align}
	R_{ijkl}&=\frac{1}{2}(2-S)(\delta_{ik}\delta_{jl}-\delta_{il}\delta_{jk}),\label{Curv1} \\
	R^{\perp}_{\alpha\beta kl}&=h_{km}^\alpha h_{ml}^\beta-h_{km}^\beta h_{ml}^\alpha,\label{Curv2}\\
	R^{\perp}_{\alpha\beta 12k}&=2(b^\beta a_k^\alpha+a^\alpha h_{12k}^\beta-b^\alpha a_k^\beta-a^\beta h_{12k}^\alpha).\label{Curv3}
\end{align}
Let $a_{\alpha\beta}=\left\langle S_\alpha,S_\beta\right\rangle$ and $A=(a_{\alpha\beta})$. Denote by $\lambda_1\ge\lambda_2\ge\cdots\ge\lambda_q$ the eigenvalues of the fundamental matrix $A$. Then $$S=\sum_{i,j,\alpha}(h^\alpha_{ij})^2=\lambda_1+\cdots+\lambda_q.$$ The Simons identity for minimal submanifolds in a unit sphere is
 \begin{equation}\label{simonidentity}
 	\frac{1}{2}\tri S=\B_1+2S-|A|^2-\rho_0^\perp.
 \end{equation}

\section{Lu's conjecture for minimal $2$-spheres}
First, we have the following lemma.
\begin{lem}\label{lem4.1}
	Let $M^2$ be a minimal surface immersed in $\mathbb{S}^{2+q}$. Then we have
	\begin{enumerate}
		\item $|A|^2=4|a|^4+4|b|^4+8\l a,b\r^2$;
		\item $\rho_0^\perp=16|a|^2|b|^2-16\l a,b\r^2$;
		\item $|\nabla S|^2=16\left(\l a,a_1\r^2+\l b,a_2\r^2+\l a,a_2\r^2+\l b,a_1\r^2+2\l a,a_1\r\l b,a_2\r-2\l a,a_2\r\l b,a_1\r\right)$;
		\item $2S^2=\rho_0^\perp+2|A|^2$.
	\end{enumerate}
\end{lem}
\begin{proof}
	By definition, we have
	\begin{equation*}
		\begin{aligned}
			|A|^2&=\l 2aa^T+2bb^T,2aa^T+2bb^T\r\\
			&=4|a|^2|a|^2+4|b|^2|b|^2+8\l a,b\r^2\\
			&=4|a|^4+4|b|^4+8\l a,b\r^2
		\end{aligned}
	\end{equation*}
	and
	\begin{equation*}
		\begin{aligned}
			\rho_0^\bot=\sum|[S_\alpha,S_\beta]|^2&=8\sum(a^\alpha b^\beta-a^\beta b^\alpha)^2\\
			&=8|a|^2|b|^2+8|a|^2|b|^2-16\l a,b\r^2\\
			&=16|a|^2|b|^2-16\l a,b\r^2.
		\end{aligned}
	\end{equation*}
	Since $S=2|a|^2+2|b|^2$, we have
	\begin{equation*}
		\begin{aligned}
			S_1&=4\left(\l a,a_1\r+\l b,a_2\r\right),\\
			S_2&=4\left(\l a,a_2\r-\l b,a_1\r\right).
		\end{aligned}
	\end{equation*}
	Hence,
	\begin{equation*}
		\begin{aligned}
			|\nabla S|^2&=S_1^2+S_2^2\\
			&=16\left(\l a,a_1\r^2+\l b,a_2\r^2+\l a,a_2\r^2+\l b,a_1\r^2+2\l a,a_1\r\l b,a_2\r-2\l a,a_2\r\l b,a_1\r\right).
		\end{aligned}
	\end{equation*}
	Finally, a straightforward calculation gives $2S^2=\rho_0^\perp+2|A|^2$, which completes the proof.
\end{proof}
Now we need to calculate the eigenvalues of $A$.
\begin{lem}\label{lem eigen of A}
	Let $M^2$ be a minimal surface immersed in $\mathbb{S}^{2+q}$. The eigenvalues of $A$ are
	\begin{equation*}
		\begin{aligned}
			\lambda_1&=\frac{S}{2}+\sqrt{\left(|a|^2-|b|^2\right)^2+4\l a,b\r^2}=\frac{S}{2}+\frac{1}{2}\sqrt{S^2-\rho_0^\perp},\\
			\lambda_2&=\frac{S}{2}-\sqrt{\left(|a|^2-|b|^2\right)^2+4\l a,b\r^2}=\frac{S}{2}-\frac{1}{2}\sqrt{S^2-\rho_0^\perp},\\
			\lambda_3&=0,\\
			&\vdots\\
			\lambda_q&=0.
		\end{aligned}
	\end{equation*}
\end{lem}
\begin{proof}
	We have $S=2\left(|a|^2+|b|^2\right)$ and $$A=\left(\sum_{i,j}h^\alpha_{ij}h^\beta_{ij}\right)_{q\times q}=2aa^T+2bb^T.$$ For any column vector $x\in\mathbb{R}^q$, we have $$Ax=2\left(\l a,x\r a+\l b,x\r b\right).$$ Then $$\operatorname{range}(A)\subset\operatorname{span}\{a,b\}\quad\text{and}\quad\operatorname{rank}(A)\le2.$$ A straightforward calculation yields that
	\begin{align*}
		Aa&=2\left(\l a,a\r a+\l a,b\r b\right)\\
		&=2\left(|a|^2a+\l a,b\r b\right),
	\end{align*}
	and
	\begin{align*}
		Ab&=2\left(\l a,b\r a+\l b,b\r b\right)\\
		&=2\left(\l a,b\r a+|b|^2b\right).
	\end{align*}
	Therefore, we have $(Aa,Ab)=(a,b)\mathcal{J}$, where
	$$
	\mathcal{J}=2
	\begin{bmatrix}
	|a|^2 & \l a,b\r \\
	\l a,b\r & |b|^2
	\end{bmatrix}.
	$$
	Thus the nonzero eigenvalues of $A$ are exactly the eigenvalues of this $2\times 2$ symmetric matrix $\mathcal{J}$. By
	\begin{equation*}
		\begin{aligned}
			0&=\det\left(2
				\begin{bmatrix}
				|a|^2 & \l a,b\r \\
				\l a,b\r & |b|^2
				\end{bmatrix}-\lambda
				\begin{bmatrix}
				1 & 0\\
				0 & 1
				\end{bmatrix}\right)\\
			&=\det\begin{bmatrix}
				2|a|^2-\lambda & 2\l a,b\r \\
				2\l a,b\r & 2|b|^2-\lambda
				\end{bmatrix}\\
			&=(2|a|^2-\lambda)(2|b|^2-\lambda)-(2\l a,b\r)^2\\
			&=\lambda^2-2(|a|^2+|b|^2)\lambda+4(|a|^2|b|^2-\l a,b\r^2)
		\end{aligned}
	\end{equation*}
	and Lemma \ref{lem4.1} we obtain
	\begin{equation*}
		\begin{aligned}
			\lambda_{1,2}&=\frac{2(|a|^2+|b|^2)\pm\sqrt{4(|a|^2+|b|^2)^2-16(|a|^2|b|^2-\l a,b\r^2)}}{2}\\
			&=|a|^2+|b|^2\pm\sqrt{(|a|^2-|b|^2)^2+4\l a,b\r^2}\\
			&=\frac{S}{2}\pm\sqrt{(|a|^2-|b|^2)^2+4\l a,b\r^2}\\
			&=\frac{S}{2}\pm\frac{1}{2}\sqrt{S^2-\rho_0^\perp}.
		\end{aligned}
	\end{equation*}
	All remaining eigenvalues are $0$ with multiplicity $q-\operatorname{rank}\{a,b\}$.
\end{proof}
Next, we need the following lemma, which
establishes a key relation between the vectors $a$ and $b$ for minimal $2$-spheres.
\begin{lem}\label{ab}
	Suppose that $M^2$ is a $2$-sphere minimally immersed in a unit sphere $\mathbb{S}^{2+q}$. We have
	\begin{enumerate}
		\item\label{abortho} $|a|^2=|b|^2=\frac{S}{4}$ and $\langle a,b\rangle=0$.
		\item\label{eigennormal} $\lambda_1=\lambda_2=\frac{S}{2}$ and the remaining eigenvalues are $0$.
	\end{enumerate}
\end{lem}
\begin{proof}
	Define $$\phi:=(|a|^2-|b|^2-2i\l a,b\r)dz^4,$$ where $z$ is the complex coordinate on the sphere. Then, $\phi$ is a differential $4$-form. It can be verified that $\phi$ is independent of the choice of the local frame, thus $\phi$ is defined globally on $M$. We now show that $\phi$ is holomorphic by verifying the Cauchy-Riemann equations. First we have
	\begin{align*}
		&\hspace{1.3em}e_1(|a|^2-|b|^2)+e_2(2\l a,b\r)\\
		&=2\l a,a_{1}\r-2\l b,b_{1}\r+2\l a,b_2\r+2\l b,a_{2}\r\\
		&=2\l a,a_{1}\r-2\l b,a_2\r-2\l a,a_1\r+2\l b,a_{2}\r\\
		&=0.
	\end{align*}
	A similar argument shows that $e_2(|a|^2-|b|^2)-e_1(2\l a,b\r)=0$. Therefore, $\phi$ is holomorphic. But any holomorphic differential form on the $2$-sphere must be zero. The result \eqref{abortho} follows. Combining \Cref{lem eigen of A}, we obtain $\lambda_1=\lambda_2=\frac{S}{2}$ and $\lambda_3=\cdots=\lambda_q=0$.
\end{proof}
\begin{rem}
	Using the same process, we can also obtain the relationship between $a_1$ and $a_2$ for minimal $2$-spheres \emph{(}see \cite{DGL}\emph{).}
\end{rem}

Furthermore, the following theorems are required.

\begin{thm}[Calabi \cite{Calabi}]\label{thm Calabi}
	Let $M^2$ be a $2$-sphere with a Riemannian metric of constant curvature $K$, and let $X:M\to r\mathbb{S}^{N}\subset\mathbb{R}^{N+1}~(N\ge2)$ be an isometric, minimal immersion of $M$ into the sphere with radius $r$, such that the image is not contained in any hyperplane of $\mathbb{R}^{N+1}$. Then
	\begin{enumerate}
 		\item The dimension $N$ is even, i.e., $N=2s$;
 		\item The value of $K$ is uniquely determined as $$K=\frac{2}{s(s+1)r^2}\eqqcolon K(s,r);$$
 		\item The immersion $X$ is uniquely determined up to a rigid rotation of $r\mathbb{S}^{N}$, and the $N$ components of the vector $X$ are a suitably normalized basis for the spherical harmonics of order $s$ on $M$.
	\end{enumerate}
\end{thm}
The immersion denoted by $\Psi_{2,s}:S^2(K(s))\to\mathbb{S}^{2s}$ is called Calabi's $2$-sphere, where $K(s)=K(s,1)$, $S^2(K(s))$ is the sphere with curvature $K(s)$, and $s=1,2,\cdots$. As we mentioned before, the Simon conjecture has only been solved in the cases $s=1$ and $s=2$. To continue, we first state the second gap theorem (cf. \cite{DGL,Simon}) of the Simon conjecture:
 	
\begin{thm}[\cite{DGL,Simon}]\label{simongap2}
	The Simon conjecture is true for $s=2$, i.e., if $\frac{4}{3}\le S\le\frac{5}{3}$, then either $S=\frac{4}{3}$ or $S=\frac{5}{3}$.
\end{thm}
Now we can give the proof of \Cref{main1}.
\begin{proof}[\textbf{Proof of \Cref{main1}}]
	Since $M$ is a closed minimal $2$-sphere, using \Cref{ab} we obtain that $\lambda_2=\frac{S}{2}$. Since $S+\lambda_2$ is constant, it follows that $S$ is constant, which implies that $K$ is constant. Therefore, \Cref{main1} \eqref{main1.1} follows from \Cref{thm Calabi}.\par
	For Theorem \ref{main1} \eqref{main1.2}, namely, in the case of $S+\lambda_2>2$, we prove as follows. By \Cref{ab}, $\lambda_2=\frac{1}{2}S$ and thus $S+\lambda_2=\frac{3}{2}S$. By setting
	\begin{equation*}
		S(s)\coloneqq 2-2K(s)=\frac{2(s-1)(s+2)}{s(s+1)},
	\end{equation*}
	if \Cref{Udo Simon} is true and $\frac{3}{2}S(s) \leq S+\lambda_2=\frac{3}{2}S\leq  \frac{3}{2}S(s+1)$ for an $s\in\mathbb{N}$, then either $S=S(s)$ or $S=S(s+1)$. In particular, we argue by contradiction and assume that $\max_{p\in M}(S+\lambda_2)<\frac{5}{2}$. Then we have $S+\lambda_2<\frac{5}{2}$, which implies that $$2<S+\lambda_2=\frac{3}{2}S<\frac{5}{2}.$$ It follows that $\frac{4}{3}<S<\frac{5}{3}$. However, such a surface does not exist by \Cref{simongap2}. Hence we obtain $\max_{p\in M} \left( S+\lambda_2\right)(p)\ge\frac{5}{2}$, and the equality holds if and only if $q=4$ and the submanifold is Calabi's $2$-sphere with curvature $K=\frac{1}{6}$ by \Cref{thm Calabi}.
\end{proof}

\section{Lu's conjecture for minimal surfaces}
We next consider the more general case. For convenience, we define
\begin{equation}\label{udef}
	\u\coloneqq S+\lambda_2=\frac{3}{2}S-\sqrt{\left(|a|^2-|b|^2\right)^2+4\l a,b\r^2}=\frac{3}{2}S-\frac{1}{2}\sqrt{S^2-\rho_0^\perp}\le\frac{3}{2}S.
\end{equation}
Since $0\le\rho_0^\perp\le S^2$ shown in \eqref{rhorho0} by the DDVV inequality, we define the function $t:M\to[0,1]$ such that
\begin{equation}\label{tdef}
	\sqrt{S^2-\rho_0^\perp}=tS.
\end{equation}
Then by \eqref{udef} we have
\begin{equation*}
	\mathfrak{u}=S+\lambda_2=\frac{1}{2}\left(3S-\sqrt{S^2-\rho_0^\perp}\right)=\frac{3-t}{2}S\le\frac{3}{2}S.
\end{equation*}
Thus
\begin{equation}\label{??}
	3S-2\mathfrak{u}=\sqrt{S^2-\rho_0^\perp}.
\end{equation}
It follows from \eqref{??} that
\begin{equation*}
	S=\frac{3}{4}\mathfrak{u}\pm\frac{\sqrt{\mathfrak{u}^2-2\rho_0^\perp}}{4}.
\end{equation*}

First we give the proofs of \Cref{main4} and \Cref{main4.5}.
\begin{proof}[\textbf{Proof of \Cref{main4}}]
	Since the universal cover of $M$ is not diffeomorphic to a $2$-sphere, it follows that $\max_{p\in M}S(p)\ge2$. If $S=\frac{3}{4}\mathfrak{u}-\frac{\sqrt{\mathfrak{u}^2-2\rho_0^\perp}}{4}$, then we obtain that $\max_{p\in M}(S+\lambda_2)(p)=\max_{p\in M}\mathfrak{u}(p)\ge\frac{8}{3}$. On the other hand, if $S=\frac{3}{4}\mathfrak{u}+\frac{\sqrt{\mathfrak{u}^2-2\rho_0^\perp}}{4}$ and there exists $p\in M$ such that $S(p)\ge 2$, then we obtain that
	\begin{equation*}
		(\mathfrak{u}(p)-3)^2+\frac{\rho_0^\perp(p)}{4}-1\le0.
	\end{equation*}
	Hence we obtain that $\rho^\perp(p)\le1$ and
	\begin{equation*}
		3-\sqrt{1-\left(\rho^\perp\right)^2(p)}\le(S+\lambda_2)(p)\le3+\sqrt{1-\left(\rho^\perp\right)^2(p)},
	\end{equation*}
	i.e.,
	\begin{equation*}
		\max_{p\in M}(S+\lambda_2)(p)\ge3-\sqrt{1-\min_{p\in M}\left(\rho^\perp\right)^2(p)},
	\end{equation*}
	which completes the proof.
\end{proof}
\begin{proof}[\textbf{Proof of \Cref{main4.5}}]
	By \Cref{lem eigen of A} we obtain
	\begin{equation*}
		\begin{aligned}
			S+\lambda_2&=S+\frac{S}{2}-\frac{1}{2}\sqrt{S^2-\rho_0^\perp}\\
			&=S+\frac{\rho_0^\perp}{2\left(S+\sqrt{S^2-\rho_0^\perp}\right)}\\
			&\ge S+\frac{\rho_0^\perp}{4S}.
		\end{aligned}
	\end{equation*}
	First, we consider the case when $M$ is orientable. By the Gauss-Bonnet formula and the Gauss equation $2K=2-S$, we obtain
	\begin{equation*}
		\begin{aligned}
			\int_M(S+\lambda_2)&\ge\int_M S+\int_M\frac{\rho_0^\perp}{4S}\\
			&=2\operatorname{Area}(M)+8\pi(g-1)+\int_M\frac{\rho_0^\perp}{4S}\\
			&\ge 2\operatorname{Area}(M)+8\pi(g-1)+\frac{1}{4\max_{p\in M}S(p)}\int_M\rho_0^\perp,
		\end{aligned}
	\end{equation*}
	where $g$ denotes the genus of $M$. Since $M$ is not diffeomorphic to a $2$-sphere, we obtain that $g\ge 1$. Thus
	\begin{equation*}
		\begin{aligned}
			\max_{p\in M}(S+\lambda_2)(p)&\ge2+\frac{1}{4\max_{p\in M}S(p)\operatorname{Area}(M)}\int_M\rho_0^\perp+\frac{8\pi(g-1)}{\operatorname{Area}(M)}\\
			&\ge2+\frac{1}{4\max_{p\in M}S(p)\operatorname{Area}(M)}\int_M\rho_0^\perp\\
			&\ge2+\frac{1}{4\max_{p\in M}(S+\lambda_2)(p)\operatorname{Area}(M)}\int_M\rho_0^\perp,
		\end{aligned}
	\end{equation*}
	which yields that
	\begin{equation*}
		\begin{aligned}
			\max_{p\in M}(S+\lambda_2)(p)&\ge1+\sqrt{1+\frac{1}{4\operatorname{Area}(M)}\int_M\rho_0^\perp}
			=1+\sqrt{1+\frac{1}{\operatorname{Area}(M)}\int_M\left(\rho^\perp\right)^2}.
		\end{aligned}
	\end{equation*}
	In general, any closed non-orientable surface is diffeomorphic to a connected sum of $\mathfrak{g}$ copies of the projective plane, where $\mathfrak{g}\ge1$ and the Euler characteristic $\chi(M)=2-\mathfrak{g}$. If $\mathfrak{g}=1$, then $M$ is diffeomorphic to $\mathbb{R}P^2$, in which case the universal cover of $M$ is diffeomorphic to a $2$-sphere. Thus it suffices to assume that $\mathfrak{g}\ge2$, which gives that $\chi(M)\le0$. Then we have
	\begin{equation*}
		\begin{aligned}
			\int_M(S+\lambda_2)&\ge\int_M\left(S+\frac{\rho_0^\perp}{4S}\right)\\
			&\ge2\operatorname{Area}(M)+\frac{1}{4\max_{p\in M}S(p)}\int_M\rho_0^\perp-4\pi\chi(M)\\
			&\ge2\operatorname{Area}(M)+\frac{1}{4\max_{p\in M}S(p)}\int_M\rho_0^\perp.
		\end{aligned}
	\end{equation*}
	Hence we obtain that
	\begin{equation*}
		\max_{p\in M}(S+\lambda_2)(p)\ge2+\frac{1}{4\max_{p\in M}(S+\lambda_2)(p)\operatorname{Area}(M)}\int_M\rho_0^\perp,
	\end{equation*}
	which implies the same lower bound as the orientable case.
\end{proof}
To prove \Cref{main3}, we need the following lemmas.
\begin{lem}[\cite{DGL}]\label{lem4.4}
	Let $M^2$ be a minimal surface immersed in $\mathbb{S}^{2+q}$. Then
	\begin{equation*}
		\begin{aligned}
			(h_{ijk}^\alpha\tri h_{ij}^\alpha)_k&=\frac{1}{2}(4-3S)\mathcal{B}_1+(2-S)^2S+\frac{1}{2}(5S-8)(-S^2+|A|^2+\rho_0^\bot)\\
			&\hspace{1.3em}-\frac{1}{4}|\nabla S|^2+32\l a,a_2\r\l b,a_1\r-32\l a,a_1\r\l b,a_2\r\\
			&=\frac{1}{2}(2-S)S^2+(2-S)(S^2-|A|^2-\rho_0^{\bot})+\tri S-\frac{3}{8}\tri S^2\\
			&\hspace{1.3em}+8(\l a,a_2\r+\l b,a_1\r)^2+8(\l a,a_1\r-\l b,a_2\r)^2.
		\end{aligned}
	\end{equation*}
\end{lem}
\begin{proof}
	See $(2.17)$ in the proof of Theorem 1 in \cite{DGL}.
\end{proof}
\begin{lem}\label{lem4.2}
	Let $M^2$ be a minimal surface immersed in $\mathbb{S}^{2+q}$. Then
	\begin{equation}
		\begin{aligned}
			\frac{1}{2}\tri\mathcal{B}_1&=h_{ijk}^\alpha\tri h_{ijk}^\alpha+\mathcal{B}_2\\
			&=\frac{1}{2}(14-9S)\mathcal{B}_1-\frac{3}{4}|\nn S|^2+72\l a,a_2\r\l b,a_1\r-72\l a,a_1\r\l b,a_2\r\\
			&\hspace{1.3em}+4\left(\l a,a_1\r^2+\l a,a_2\r^2+\l b,a_1\r^2+\l b,a_2\r^2\right)+\mathcal{B}_2\\
			&=\frac{1}{2}(14-9S)\B_1-\frac{7}{4}|\nn S|^2+\B_2+20\ll(\l a,a_1\r-\l b,a_2\r\rr)^2+20\ll(\l a,a_2\r+\l b,a_1\r\rr)^2.
		\end{aligned}
	\end{equation}
\end{lem}
\begin{proof}
	By \eqref{Lap2}, we have
	\begin{equation*}\label{Sum1}
		\begin{aligned}
			h_{ijk}^\alpha\tri h_{ijk}^\alpha &=(h_{ijk}^\alpha\tri h_{ij}^\alpha)_k-\sum(\tri h_{ij}^\alpha)^2+2h_{ijk}^\alpha h_{pj}^\alpha R_{pikmm}+h_{ijk}^\alpha h_{ijp}^\alpha R_{pmkm}\\
			&\hspace{1.3em}+4h_{ijk}^\alpha h_{pjm}^\alpha R_{pikm}+2h_{ijk}^\alpha h_{ijm}^\beta R^{\perp}_{\beta\alpha km}+h_{ijk}^\alpha h_{ij}^\beta R^{\perp}_{\beta\alpha kmm}.
		\end{aligned}
	\end{equation*} 
	By \eqref{Curv1}, it follows that
	\begin{equation*}\label{sumpinch2term3}
		\begin{aligned}
			2h_{ijk}^\alpha h_{pj}^\alpha R_{pikmm}&=-h_{ijk}^\alpha h_{pj}^\alpha(\delta_{pk}\delta_{im}-\delta_{pm}\delta_{ik})S_m\\
			&=-h_{ijm}^\alpha h_{ij}^\alpha S_m\\
			&=-\frac{1}{2}|\nabla S|^2,
		\end{aligned}
	\end{equation*} 
	and
	\begin{equation*}\label{sumpinch2term45}
		4h_{ijk}^\alpha h_{pjm}^\alpha R_{pikm}+h_{ijk}^\alpha h_{ijp}^\alpha R_{pmkm}=5\ll(1-\frac{S}{2}\rr)\mathcal{B}_1.
	\end{equation*} 
	By \eqref{Curv2}, we obtain
	\begin{equation*}\label{sumpinch2term6}
		\begin{aligned}
			2h_{ijk}^\alpha h_{ijm}^\beta R^{\perp}_{\beta\alpha km}&=16(a_1^\alpha a_2^\beta-a_2^\alpha a_1^\beta)(a^\beta b^\alpha-a^\alpha b^\beta)\\
			&=32\ll(\l a,a_2\r\l b,a_1\r-\l a,a_1\r\l b,a_2\r\rr).
		\end{aligned}
	\end{equation*}
	By \eqref{Curv3}, we get
	\begin{equation*}\label{Last}
		\begin{aligned}
			h_{ijk}^\alpha h_{ij}^\beta R^{\perp}_{\beta\alpha kmm}&=h_{ij1}^\alpha h_{ij}^\beta R^{\perp}_{\beta\alpha 122}-h_{ij2}^\alpha h_{ij}^\beta R^{\perp}_{\beta\alpha 121}\\
			&=4(a_1^\alpha a^\beta+a_2^\alpha b^\beta)(a_2^\beta b^\alpha-a^\beta a_1^\alpha-a_2^\alpha b^\delta+a^\alpha a_1^\beta)\\
			&\hspace{1.3em}-4(a_2^\alpha a^\beta-a_1^\alpha b^\beta)(a_1^\beta b^\alpha+a^\beta a_2^\alpha-a_1^\alpha b^\beta-a^\alpha a_2^\beta)\\
			&=-\frac{1}{2}S\mathcal{B}_1+8\l a,a_2\r\l b,a_1\r-8\l a,a_1\r\l b,a_2\r\\
			&\hspace{1.3em}+4\ll(\l a,a_1\r^2+\l a,a_2\r^2+\l b,a_1\r^2+\l b,a_2\r^2\rr).
		\end{aligned}
	\end{equation*}
	Finally,
	\begin{equation*}
		\begin{aligned}
			(h_{ijk}^\alpha\tri h_{ij}^\alpha)_k
			-\sum(\tri h_{ij}^\alpha)^2&=h_{ijk}^\alpha h_{pik}^\alpha R_{pljl}+h_{ijk}^\alpha h_{lpk}^\alpha R_{pijl}+h_{ijk}^\alpha h_{pi}^\alpha R_{pljlk}\\
			&\hspace{1.3em}+h_{ijk}^\alpha h_{lp}^\alpha R_{pijlk}+h_{ijk}^\alpha h_{lik}^\delta R^{\perp}_{\delta\alpha jl}+h_{ijk}^\alpha h_{li}^\delta R^{\perp}_{\delta\alpha jlk}\\
			&=(2-S)\mathcal{B}_1+16\l a,a_2\r\l b,a_1\r-16\l a,a_1\r\l b,a_2\r-\frac{1}{2}|\nabla S|^2\\
			&\hspace{1.3em}-\frac{1}{2}S\mathcal{B}_1+\frac{1}{4}|\nabla S|^2+16\l a,a_2\r\l b,a_1\r-16\l a,a_1\r\l b,a_2\r\\
			&=(2-S)\mathcal{B}_1-\frac{1}{4}|\nabla S|^2-\frac{1}{2}S\mathcal{B}_1+32\l a,a_2\r\l b,a_1\r-32\l a,a_1\r\l b,a_2\r.
		\end{aligned}
	\end{equation*}
	Combining the above calculations, we obtain
	\begin{equation*}\label{SecSum}
		\begin{aligned}
			\frac{1}{2}\tri\B_1&=h_{ijk}^\alpha\tri h_{ijk}^\alpha+\B_2\\
			&=\frac{1}{2}(14-9S)\mathcal{B}_1-\frac{3}{4}|\nn S|^2+72\l a,a_2\r\l b,a_1\r-72\l a,a_1\r\l b,a_2\r\\
			&\hspace{1.3em}+4\left(\l a,a_1\r^2+\l a,a_2\r^2+\l b,a_1\r^2+\l b,a_2\r^2\right)+\mathcal{B}_2.
		\end{aligned}
	\end{equation*}
	Finally, by \Cref{lem4.1} we obtain
	\begin{equation*}
		\begin{aligned}
			\frac{1}{2}\tri\mathcal{B}_1&=\frac{1}{2}(14-9S)\B_1-\frac{3}{4}|\nabla S|^2+\B_2+4\ll(\l a,a_1\r^2+\l b,a_2\r^2+\l a,a_2\r^2+\l b,a_1\r^2\rr)\\
			&\hspace{1.3em}+72\ll(\l a,a_2\r\l b,a_1\r-\l a,a_1\r\l b,a_2\r\rr)\\
			&=\frac{1}{2}(14-9S)\B_1-\frac{7}{4}|\nn S|^2+\B_2+20\ll(\l a,a_1\r-\l b,a_2\r\rr)^2+20\ll(\l a,a_2\r+\l b,a_1\r\rr)^2.
		\end{aligned}
	\end{equation*}
	which proves the lemma.
\end{proof}
\begin{lem}\label{lem4.3}
	Let $M^2$ be a minimal surface immersed in $\mathbb{S}^{2+q}$. Then $$\B_2=S(2-S)^2-\frac{8-5S}{4}\rho_0^\perp+\C_1,$$ where $\mathcal{C}_1=2\left(|a_{11}-a_{22}|^2+|a_{12}+a_{21}|^2\right)$. 
\end{lem}
\begin{proof}
	By definition, we have
	\begin{equation*}
		\begin{aligned}
			\mathcal{B}_2&=4\left(|a_{11}|^2+|a_{22}|^2+|a_{12}|^2+|a_{21}|^2\right)\\
			&=2\left(|a_{11}+a_{22}|^2+|a_{12}-a_{21}|^2\right)+2\left(|a_{11}-a_{22}|^2+|a_{12}+a_{21}|^2\right)\\
			&=2\left(|\tri a|^2+|\tri b|^2\right)+2\left(|a_{11}-a_{22}|^2+|a_{12}+a_{21}|^2\right)\\
			&=2\left(|\tri a|^2+|\tri b|^2\right)+\mathcal{C}_1,
		\end{aligned}
	\end{equation*}
	where $\mathcal{C}_1=2\left(|a_{11}-a_{22}|^2+|a_{12}+a_{21}|^2\right)$. Here, by \eqref{Lap1} we have
	\begin{align*}
		\tri a&=a(2-S)+2b\l a,b\r-2a|b|^2,\\
		\tri b&=b(2-S)+2a\l a,b\r-2b|a|^2.
	\end{align*}
	Then by Lemma \ref{lem4.1}, we have
	\begin{align*}
		|\tri a|^2&=(2-S)^2|a|^2+\frac{1}{4}|b|^2\rho_0^\perp-\frac{1}{4}(2-S)\rho_0^\perp,\\
		|\tri b|^2&=(2-S)^2|b|^2+\frac{1}{4}|a|^2\rho_0^\perp-\frac{1}{4}(2-S)\rho_0^\perp.
	\end{align*}
	Hence, $$|\tri a|^2+|\tri b|^2=\frac{1}{2}S(2-S)^2-\frac{1}{8}(8-5S)\rho_0^\perp.$$ Therefore, we obtain
	\begin{equation}
		\begin{aligned}
			\mathcal{B}_2&=2\left(|\tri a|^2+|\tri b|^2\right)+\mathcal{C}_1\\
			&=S(2-S)^2-\frac{8-5S}{4}\rho_0^\perp+\mathcal{C}_1,
		\end{aligned}
	\end{equation}
	which proves the lemma.
\end{proof}
\begin{proof}[\textbf{Proof of \Cref{main3}}]
	By \eqref{udef}, we have $$(S-2\lambda_2)^2=S^2-\rho_0^\perp.$$ Using the Simons identity \eqref{simonidentity} and Lemma \ref{lem4.1}, we obtain
	\begin{equation}\label{new4.2}
		\begin{aligned}
			\frac{1}{2}\tri S&=\mathcal{B}_1+2S-|A|^2-\rho_0^\perp\\
			&=\mathcal{B}_1+2S-S^2-\frac{1}{2}\rho_0^\perp\\
			&=\mathcal{B}_1+2S-\frac{3}{2}S^2+2\left(\frac{3}{2}S-\u\right)^2\\
			&=\B_1-\frac{1}{2}S(3S-4)+\frac{1}{2}(S-2\lambda_2)^2.
		\end{aligned}
	\end{equation}
	Integrating over $M$ on both sides of the equation and combining $\B_1\ge0$, we prove \eqref{1stgap}. To prove \eqref{2ndgap}, using the Simons identity \eqref{simonidentity}, by integration by parts and compactness of $M$, we obtain $$\int_M-\frac{1}{2}|\nabla S|^2=\int_M\frac{1}{2}S\tri S=\int_M\left(S\mathcal{B}_1+(2-S)S^2-\frac{S}{2}\rho_0^\perp\right).$$ Then, combining with \eqref{new4.2}, we have
	\begin{equation}\label{000}
		\begin{aligned}
			\int_M\frac{1}{2}(14-9S)\mathcal{B}_1&=\int_M\left(7\mathcal{B}_1-\frac{9}{2}S\mathcal{B}_1\right)\\
			&=\int_M\left(7S^2-14S+\frac{7}{2}\rho_0^\perp+\frac{9}{4}|\nabla S|^2+\frac{9}{2}(2-S)S^2-\frac{9}{4}S\rho_0^\perp\right)\\
			&=\int_M\left(\frac{9}{4}|\nabla S|^2+16S^2-\frac{9}{2}S^3-14S+\frac{1}{4}(14-9S)\rho_0^\perp\right).
		\end{aligned}
	\end{equation}
	Integrating over $M$ on both sides of \Cref{lem4.4} gives that
	\begin{equation}\label{222}
		\begin{aligned}
			\int_M(2-S)\ll(-\frac{3}{2}S^2+|A|^2+\rho_0^\perp\rr)&=\int_M\ll[8\ll(\l a,a_1\r-\l b,a_2\r\rr)^2+8\ll(\l a,a_2\r+\l b,a_1\r\rr)^2\rr]\\
			&=\int_M(2-S)\ll(-\frac{3}{2}S^2+\frac{2S^2-\rho^\perp}{2}+\rho_0^\perp\rr)\\
			&=\int_M(2-S)\ll(-\frac{1}{2}S^2+\frac{1}{2}\rho_0^\perp\rr).
		\end{aligned}
	\end{equation}
	Hence, combining \Cref{lem4.2}, \Cref{lem4.3} with \eqref{000} and \eqref{222} we obtain
	\begin{equation}\label{qqq}
		\begin{aligned}
			0&=\int_M\frac{1}{2}\tri\B_1\\
			&=\int_M\ll[\frac{1}{2}(14-9S)\B_1-\frac{7}{4}|\nabla S|^{2}+\B_2+\frac{5}{2}(2-S)\ll(-\frac{1}{2}S^{2}+\frac{1}{2}\rho_0^{\perp}\rr)\rr]\\
			&=\int_M\ll[\frac{1}{2}|\nabla S|^{2}+\frac{S}{2}(10-7S)(S-2)+\frac{1}{2}(3-2S)\rho_0^{\perp}+\C_1+\frac{5}{2}(2-S)\ll(\frac{1}{2}\rho_0^{\perp}-\frac{1}{2}S^{2}\rr)\rr]\\
			&=\int_M\ll[\frac{1}{2}|\nabla S|^{2}+\frac{S}{4}(S-2)(20-9S)+\frac{1}{4}\rho_0^{\perp}(16-9S)+\C_1\rr]\\
			&=\int_M\ll[\frac{1}{2}|\nabla S|^2-\frac{S}{2}(3S-4)(3S-5)-(16-9S)\ll(\frac{3}{2}S-\u\rr)^2+\C_1\rr].
		\end{aligned}
	\end{equation}
	Thus, combining \Cref{lem4.3}, $\C_1\ge0$ and \eqref{qqq} we obtain
	\begin{equation}
		\begin{aligned}
			0&\le\int_M\left[2\C_1+|\nn S|^2\right]\\
			&=\int_M\left[S(3S-4)(3S-5)+\frac{1}{2}(16-9S)(S-2\lambda_2)^2\right]\\
			&=\int_M\ll[\frac{S}{2}(S-2)(9S-20)+2(\rho^\perp)^2(9S-16)\rr]\\
			&=2\int_M\ll[\mathcal{B}_2+\frac{1}{2}|\nabla S|^2-S(2-S)^2+(8-5S)(\rho^\perp)^2\rr],
		\end{aligned}
	\end{equation}
	which proves \eqref{2ndgap}.
\end{proof}
Now we can give the proofs of \Cref{main5} and \Cref{main6}.
\begin{proof}[\textbf{Proof of \Cref{main5}}]
	First, by using the Gauss equation $2K=2-S$, we obtain $$\rho_0^\perp\le-\gamma K=-\frac{1}{2}\gamma(2-S)=\frac{\gamma}{2}(S-2).$$ 
	Then by \Cref{main3} we have
	\begin{equation*}
		\begin{aligned}
			0&\le\int_M\left[\frac{1}{4}S(S-2)(9S-20)+\frac{1}{4}\rho_0^\perp(9S-16)\right]\\
			&\le\int_M\frac{S-2}{4}\ll[9S^2-20S+\frac{\gamma}{2}(9S-16)\rr]\\
			&=\int_M\frac{S-2}{4}\ll[S\ll(9S-20+\frac{9}{2}\gamma\rr)-8\gamma\rr]\\
			&=\int_M\frac{S-2}{4}\ll(S-S_0\rr)\ll(S-S_0'\rr),
		\end{aligned}
	\end{equation*}
	where $S_0=\frac{40-9\gamma+\sqrt{81\gamma^2+432\gamma+1600}}{36}$ and $S_0'=\frac{40-9\gamma-\sqrt{81\gamma^2+432\gamma+1600}}{36}<0$ for $0\le\gamma\le 4$. Hence, by $2\le S\le S_0$, we obtain that $S\equiv2$ or $S\equiv S_0$. By the classification of Bryant \cite{Bryant}, there is no minimal surface in spheres with $S\equiv S_0>2$. Therefore, $S\equiv2$ and $M$ is the Clifford torus $\mathbb{S}^1(\sqrt{1/2})\times\mathbb{S}^1(\sqrt{1/2})$. To prove \eqref{main5.2}, it suffices to assume that $\mathfrak{u}=S+\lambda_2\le\frac{20}{9}$, which yields that $20-9S\ge0$. It follows by $\rho^\perp\le\frac{1}{2}\sqrt{(S+\lambda_2-2)\gamma S}$ and \eqref{rhorho0} that
	\begin{equation}\label{main5.2int}
		\begin{aligned}
			0&\le\int_M\left[\frac{S}{2}(S-2)(9S-20)+\frac{1}{2}\rho_0^\perp(9S-16)\right]\\
			&=\frac{1}{2}\int_M\left[S(\mathfrak{u}-2)(9S-20)+\rho_0^\perp(9S-16)+\frac{S(20-9S)\rho_0^\perp}{2S+2\sqrt{S^2-\rho_0^\perp}}\right]\\
			&\le\frac{1}{2}\int_M\left[S(\mathfrak{u}-2)(9S-20)+\rho_0^\perp(9S-16)+\frac{20-9S}{2}\rho_0^\perp\right]\\
			&=\frac{1}{2}\int_M\left[S(\mathfrak{u}-2)(9S-20)+\frac{1}{2}\rho_0^\perp(9S-12)\right]\\
			&\le\frac{1}{2}\int_M\left[S(\mathfrak{u}-2)\left(\frac{1}{2}(18+9\gamma)S-(20+6\gamma)\right)\right].
		\end{aligned}
	\end{equation}
	Suppose that $S\le\frac{40+12\gamma}{18+9\gamma}$ for some $0\le\gamma\le\frac{2}{3}$, it follows by \eqref{main5.2int} and the assumption $\u>2$ that $S\equiv\frac{40+12\gamma}{18+9\gamma}\ge2$ for some $0\le\gamma\le\frac{2}{3}$. If $0\le\gamma<\frac{2}{3}$, then $S\equiv\frac{40+12\gamma}{18+9\gamma}>2$. Similarly, there is no minimal surface in spheres with $S\equiv\frac{40+12\gamma}{18+9\gamma}>2$ by Bryant \cite{Bryant}. Hence, we obtain that $\gamma=\frac{2}{3}$, $S\equiv2$ and $M$ is the Clifford torus $\mathbb{S}^1(\sqrt{1/2})\times\mathbb{S}^1(\sqrt{1/2})$. Thus we have $\lambda_2=0$ and $S+\lambda_2\equiv 2$, which contradicts with $S+\lambda_2>2$. Therefore, we obtain that $$\max_{p\in M}(S+\lambda_2)(p)\ge\max_{p\in M}S(p)>\frac{40+12\gamma}{18+9\gamma}\ge2,$$ which completes the proof.
\end{proof}

\begin{proof}[\textbf{Proof of \Cref{main6}}]
	If $\rho^\perp\ge\frac{1}{2}\sqrt{1-\tau^2}S$ for some $0\le\tau\le 1$, then we obtain that  the function t defined in \eqref{tdef} satisfies $$0\le t\le\tau\le1.$$ It follows that
	\begin{equation*}
		\begin{aligned}
			\max_{p\in M}(S+\lambda_2)(p)&\ge\frac{1}{\operatorname{Area}(M)}\int_M\left(S+\lambda_2\right)\\
			&=\frac{1}{\operatorname{Area}(M)}\int_M\frac{1}{2}(3-t)S\\
			&\ge\frac{1}{\operatorname{Area}(M)}\int_M\frac{1}{2}(3-\tau)S\\
			&=\frac{3-\tau}{2\operatorname{Area}(M)}\left(2\operatorname{Area}(M)-4\pi\chi(M)\right)\\
			&=(3-\tau)\left(1-\frac{2\pi\chi(M)}{\operatorname{Area}(M)}\right).
		\end{aligned}
	\end{equation*}
	On the other hand, if $\rho^\perp\le\frac{1}{2}\sqrt{1-\tau^2}S$ for some $\frac{\sqrt{9+3\sqrt{5}}}{4}\le\tau\le 1$, then we obtain that $\frac{\sqrt{9+3\sqrt{5}}}{4}\le \tau\le t\le 1$. It follows by \Cref{main3} that
	\begin{equation}\label{star}
		\begin{aligned}
			0&\le\int_M\left[S(3S-4)(3S-5)+\frac{1}{2}(16-9S)(S-2\lambda_2)^2\right]\\
			&=\int_M\frac{S}{2}\left(18-9t^2\right)\left[\left(S-\frac{27-8t^2}{18-9t^2}\right)^2+\frac{\frac{45}{4}-\left(8t^2-\frac{9}{2}\right)^2}{\left(18-9t^2\right)^2}\right]\\
			&=\int_M\frac{S}{2}\left(18-9t^2\right)(S-T_B(t))(S-T_A(t))\\
			&=\int_M\frac{S}{2}\left(18-9t^2\right)\frac{4}{(3-t)^2}\left(\mathfrak{u}-\frac{3-t}{2}T_B(t)\right)\left(\mathfrak{u}-\frac{3-t}{2}T_A(t)\right)\\
			&=\int_M\frac{4S}{2(3-t)^2}\left(18-9t^2\right)\left(\mathfrak{u}-\hat{T}_B(t)\right)\left(\mathfrak{u}-\hat{T}_A(t)\right).
		\end{aligned}
	\end{equation}
	Thus it is impossible to satisfy
	\begin{equation*}
		\hat{T}_B(t)<\mathfrak{u}=S+\lambda_2<\hat{T}_A(t).
	\end{equation*}
	Since $\hat{T}_B(t)$ is monotonically decreasing and $\hat{T}_A(t)$ is monotonically increasing for $\tau\le t\le 1$, it follows from the assumption $S+\lambda_2>\hat{T}_B(\tau)$ and \eqref{star} that there exists $p_0\in M$ such that $(S+\lambda_2)(p_0)\ge\hat{T}_A(\tau)$, i.e., $$\max_{p\in M}(S+\lambda_2)(p)\ge\hat{T}_A(\tau)=\hat{T}_B(\tau)+\sigma(\tau),$$ where $$\sigma(\tau)=\hat{T}_A(\tau)-\hat{T}_B(\tau)=\frac{(3-\tau)\sqrt{\left(8\tau^2-\frac{9}{2}\right)^2-\frac{45}{4}}}{18-9\tau^2}$$ is a monotonically increasing function of $\frac{\sqrt{9+3\sqrt{5}}}{4}\le\tau\le 1$. In particular, if $M$ is a minimal surface with flat normal bundle, we have $\rho^\perp=\rho_0^\perp=0$. Then, using \Cref{lem eigen of A} we obtain $$\lambda_2=\frac{1}{2}S-\frac{1}{2}\sqrt{S^2-\rho_0^\perp}=0.$$ Thus we can choose $\tau=1$, which implies that $\hat{T}_B(\tau)=\hat{T}_B(1)=2$. If $S+\lambda_2>2$, it follows that $$\max_{p\in M}(S+\lambda_2)(p)\ge\hat{T}_A(1)=\frac{20}{9}.$$ Suppose that $\max_{p\in M}(S+\lambda_2)(p)=\max_{p\in M}S(p)\le\frac{20}{9}$, then $S+\lambda_2=S\le\frac{20}{9}$. Since $\hat{T}_B(1)=2<S\le\frac{20}{9}=\hat{T}_A(1)$, by \eqref{star} we obtain $S\equiv\frac{20}{9}$. However, by the classification in \cite{Bryant}, there is no minimal surface in spheres with $S\equiv\frac{20}{9}$. Therefore, $$\max_{p\in M}(S+\lambda_2)(p)=\max_{p\in M}S(p)>\frac{20}{9},$$ which completes the proof.
\end{proof}

\begin{acknow}
The authors are grateful to Professor Zhiqin Lu (The University of California, Irvine) for his constant encouragement and support.
\end{acknow}

\end{document}